\documentclass[journal]{IEEEtran}
\usepackage{amsmath}
\usepackage{graphicx}
\usepackage{amsfonts,amssymb,amsthm}
\usepackage{url}
\usepackage{comment}
\usepackage{bm}
\newcommand{\x}{\bm{x}}

\newcommand{\g}{\bm{g}}
\newcommand{\z}{\bm{z}}
\newcommand{\Rbb}{\mathbb{R}}
\newcommand{\sumj}{{\textstyle \sum_j}}
\newcommand{\argmax}{\mathop{\rm arg~max}\limits}
\newcommand{\argmin}{\mathop{\rm arg~min}\limits}

\newtheorem{PROP}{Proposition}
\newtheorem{ALGO}{Algorithm}
\newtheorem{DEFI}{Definition}
\newtheorem{LEMM}{Lemma}
\newtheorem{ASSU}{Assumption}

\begin{document}
\title{A New Dynamic Pricing Method Based on \\Convex Hull Pricing}

\author{
    Naoki~Ito,~\IEEEmembership{Non~Member,~IEEE,}
    Akiko~Takeda,~\IEEEmembership{Non~Member,~IEEE}
    and~Toru~Namerikawa,~\IEEEmembership{Member,~IEEE}
\thanks{N. Ito is with the School of Open and Environment System, Keio University, 3-14-1 Hiyoshi, Kohoku-ku, Yokohama, Kanagawa, 223-8522, Japan (e-mail: pico@a8.keio.jp).}
\thanks{A. Takeda is with the Department of Mathematical Informatics, The University of Tokyo, 7-3-1 Hongo, Bunkyo-ku, Tokyo, 113-8656, Japan (email: takeda@mist.i.u-tokyo.ac.jp).}
\thanks{T. Namerikawa is with the Department of System Design Engineering, Keio University, 3-14-1 Hiyoshi, Kohoku-ku, Yokohama, Kanagawa, 223-8522, Japan (email: namerikawa@sd.keio.ac.jp).}
\thanks{The authors are with JST, CREST, 4-1-8 Honcho, Kawaguchi, Saitama, 332-0012, Japan.}
\thanks{The earlier version \cite{Ito2013} of this paper was presented at the 2013 IEEE SmartGridComm.}
}
\markboth{}%
{Shell \MakeLowercase{\textit{et al.}}: Bare Demo of IEEEtran.cls for Journals}

\maketitle
\begin{abstract}
This paper presents a new dynamic pricing model (a.k.a. real-time pricing) that 
reflects startup costs of generators. 
Dynamic pricing, which is a method to control demand by pricing electricity at hourly 
(or more often) intervals, has been studied by many researchers. 
They assume that the cost functions of suppliers are convex, 
although they may be nonconvex because of the startup costs of generators in practice. 
We provide a dynamic pricing model that takes into account such cost functions 
within the settings of unit commitment problems (UCPs). 
Our model gives convex hull price (CHP), 
which has not been used in the context of dynamic pricing, 
though it is known that the CHP minimizes the uplift payment 
which is disadvantageous to suppliers for a given demand. 
In addition, we apply an iterative algorithm based on the subgradient method 
to solve our model. Numerical experiments show the efficiency of our model 
on reducing uplift payments. The prices determined 
by our algorithm give sufficiently small uplift payments 
in a realistic computational time. 
\end{abstract}


\IEEEpeerreviewmaketitle

\section{Introduction}
  \IEEEPARstart{D}{ynamic} pricing (a.k.a. real-time pricing) is 
  a method of invoking a response in demand 
  by pricing electricity at hourly (or more often) intervals. 
  There are many studies about dynamic pricing. 
  For example, Roozbehani et al. \cite{Roozbehani2011} 
  proposed a nonlinear control model in a real-time market, 
  in which prices are updated on moment-to-moment basis. 
  They focused on the stability of the market, 
  and analyzed stabilizing effects of their model 
  by using volatility measures of the prices.
  On the other hand, Miyano and Namerikawa \cite{Miyano2012a} proposed 
  a dynamic pricing model in a day-ahead market, in which a market operator 
  sets next day's hourly (or more often) prices. 
  Their price is given by the Lagrange multiplier of a social welfare maximization problem.
  They studied an algorithm based on a steepest descent method 
  to control the load levels, and showed its convergence. 
  Their numerical results show the efficiency of their algorithm.

  These studies assume that the cost functions of suppliers are convex, 
  although they may be nonconvex because of the startup costs of generators 
  in practice (in fact, there are many studies that deal such cost functions 
  within the settings of unit commitment problems (UCPs)). Thus, these models 
  do not fully reflect the startup costs to the prices, 
  and would be disadvantageous to the suppliers. 
  One of the measures showing the disadvantages is the {\em uplift payment} which is 
  the gap between the suppliers' optimal profit and actual profit. 
  If the cost functions are truly convex, a marginal cost price can make the uplift payments zero. 
  However, if not, none of pricing models may make the uplift payments zero. 
  The uplift payments for owners of many generators tend to be relatively small 
  and can often be ignored. However, this may not be true for small producers 
  since startup costs occupy a large portion of the total cost of electricity generation. 

  Recently, under the assumption that the demand is given 
  (i.e., out of the context of the dynamic pricing), 
  several pricing models \cite{Hogan2003,Gribik2007,Bjorndal2008,Zhang2009,O'Neill2005} 
  have been proposed in order to reduce the uplift payments. The most successful pricing model 
  is the convex hull pricing (CHP) model (a.k.a. extended locational marginal pricing model) 
  proposed by Gribik et al. \cite{Gribik2007}. The authors theoretically showed that the CHP is given 
  by the Lagrange multiplier for the UCP, and the CHP minimizes the uplift payment for a given demand. 
  Many researchers have studied algorithms to calculate the CHP, e.g., 
  \cite{Wang2013c,Wang2010}. However, they have not been used in the context of dynamic pricing.
  
  This paper presents a new dynamic pricing model based on the CHP. 
  First, we formulate a social welfare maximization problem, 
  which maximizes the sum of the consumers' utility and the suppliers' profit 
  under the condition that supply and demand are equal, 
  with nonconvex cost functions of the suppliers within the settings of the UCP. 
  Then we applied a CHP approach, which is invented for the UCP, 
  to the social welfare maximization problem; 
  our model takes into account both the startup costs and dynamic demand. 
  We prove that the CHP is given by a solution of its dual problem, 
  i.e., the Lagrange multiplier for the social welfare maximization problem. 
  This implies that our price minimizes 
  the uplift payment for the equilibrium demand. 
  Since our pricing model has a nonsmooth objective function 
  including 0-1 integer variables, it is difficult to be solved by exact optimization algorithms. 
  Thus we provide an approximate pricing algorithm based on the subgradient method. 
  Numerical results show that our pricing model 
  leads to smaller uplift payments compared with standard dynamic pricing models. 
  Moreover, our algorithm achieves sufficiently small uplift payments 
  in a realistic number of iterations and computational time. 
  
  The rest of this paper is organized as follows. 
  Section \ref{Market Model} presents the setting of a dynamic electricity market. 
  Section \ref{Convex Hull Price} presents the definition of the UCP and the CHP. 
  Section \ref{Model and Algorithm} introduces a new dynamic pricing model 
  based on a social welfare maximization that includes the UCP. 
  In addition, we theoretically show that our pricing model leads to the CHPs. 
  We also provide a pricing algorithm based on the subgradient method. 
  Numerical results are reported in Section \ref{Numerical Simulations}. 
  We give conclusions and list possible directions for research 
  in Section \ref{Conclusion}.

  In what follows, we denote column vectors in boldface, e.g., $\x\in\Rbb^n$ whose $i$-th element is $x_i\in\Rbb~(i=1,2,\dots,n)$.

\section{Market Model}\label{Market Model}
  We will begin by describing the electricity market model 
  and existing pricing models. 
  Assumptions in this section are made 
  along the lines of \cite{Roozbehani2011,Miyano2012a}. 
  There are three kinds of participants in an electricity market: 
  consumers, suppliers, and an independent system operator (ISO). 
  The suppliers (or consumers) decide their electric power production 
  (or consumption) so as to maximize their profit 
  (or utility, respectively) at a given electricity price. 
  The ISO is a non-profit institution that is independent of 
  the suppliers and the consumers. The ISO makes hourly (or more often) pricing 
  decisions to balance supply and demand. 
  We assume that resistive losses can be ignored. 
  Further, there are no line capacity constraints 
  and reserve capacity constraints. 
  
  \subsection{Supply and demand models}
  To make our model simple, we suppose 
  a single representative supplier (or consumer) whose response represents 
  the macro behavior of all suppliers (or consumers), and focus on a 
  single-period model that deals with each period independently (i.e., 
  a model with that does not take into account dynamical changes). 
  Our model and algorithm can be extended to 
  a multi-agent and multi-period model, as is shown in \cite{Hogan2003}. 
  Let $u:~[0,\infty)\rightarrow[0,\infty)$ be the utility function 
  of the representative consumer, which represents the dollar value of 
  consuming electricity. 
  Let $v:~[0,\infty)\rightarrow[0,\infty]$ be the cost function of 
  the representative supplier, which represents the dollar cost of 
  producing electricity. 
  For a given price $p>0$, the consumer (or supplier) 
  makes their demand $d(p)$ (or supply $s(p)$) 
  to maximize their utility (or profit, respectively), 
  i.e., 
  \begin{align}
    d(p) &\in \argmax_{d\geq 0}~ u(d) - pd,\label{consumer's_problem}\\
    s(p) &\in \argmax_{y\geq 0}~ py - v(y).\label{supplier's_problem}
  \end{align}
  \begin{figure}
    \centering
    \includegraphics[width=.8\columnwidth]{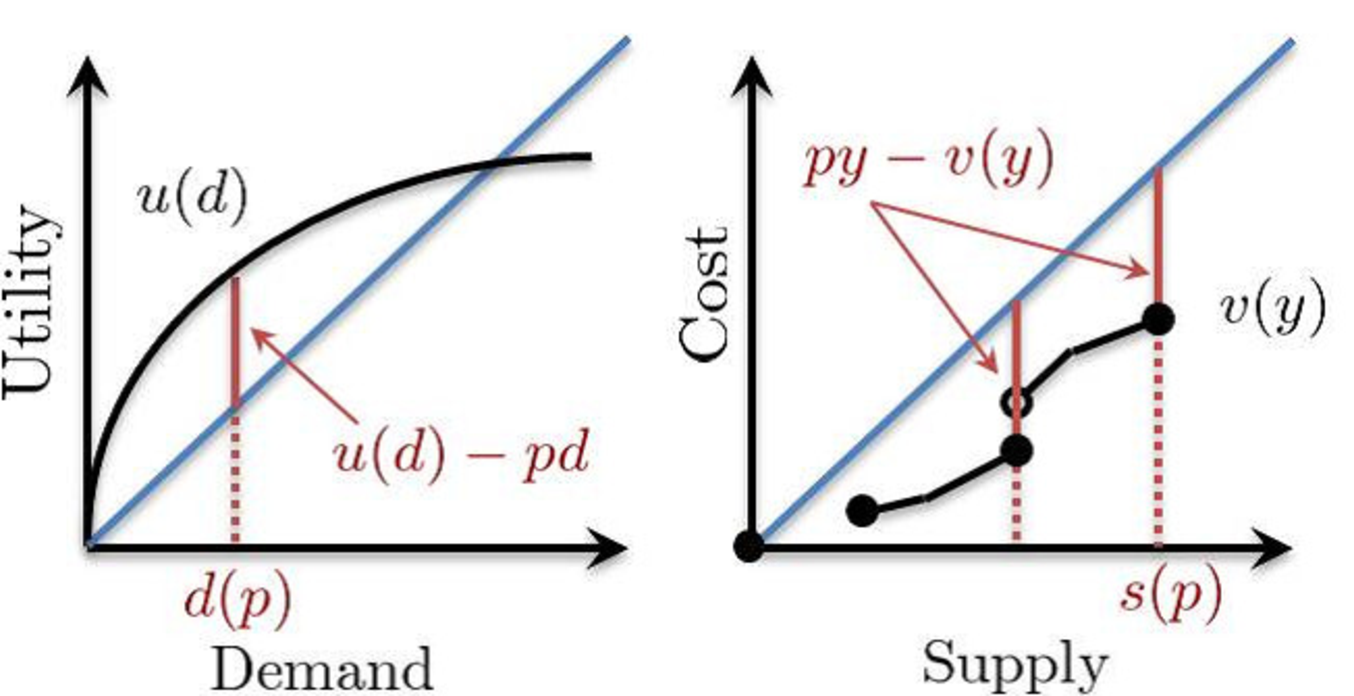}
    \caption{Illustration of supply and demand models. The left panel illustrates 
    the consumer's problem \eqref{consumer's_problem}, and the right panel illustrates 
    the supplier's problem \eqref{supplier's_problem}. 
    The vertical solid lines represent optimal utility and profit for each agent. 
    In this case, the supplier has two optimal solutions. The cost function $v$ 
    in the right panel is nonconvex and discontinuous. 
    We deal with such cost functions in this paper.}\label{sdmodel}
  \end{figure}
  A conceptual illustration is given in Fig. \ref{sdmodel}. 
  The existing works \cite{Roozbehani2011,Miyano2012a} assume 
  the following assumption:
  \begin{ASSU}[\cite{Roozbehani2011,Miyano2012a}]\label{convex_concave}
    The utility function $u$ is class-$\mathcal{C}^2$, monotonically increasing, 
    and strictly concave on $[0,\infty)$. 
    The cost function $v$ is class-$\mathcal{C}^2$, monotonically increasing, 
    and strictly convex on $[0,\infty)$. 
  \end{ASSU}
  Note that if there exist a solution of \eqref{consumer's_problem} (or \eqref{supplier's_problem}), 
  it is unique under Assumption \ref{convex_concave}.
  
  \subsection{Pricing model}
  The ISO is a non-profit institution and independent 
  from the consumer and the supplier. The objective of the ISO is 
  to manage the electricity market, especially, to balance 
  demand and supply. In the electricity market, 
  neither the consumer and the supplier do not bid. 
  Thus, the ISO should match the levels of demand $d(p)$ 
  and supply $s(p)$ by making an appropriate pricing decision $p$. 
  Here, it is natural to assume that the utility function $u$ of 
  the consumer is unknown to the ISO, while the cost function $v$ of 
  the supplier is not necessarily known to the ISO. 
  Accordingly, an electricity price $p$ is determined 
  through the following procedure: 
  \begin{ALGO}\label{pricing_procedure}
    \begin{enumerate}
      \item The ISO sets the initial price $p^0$, 
      and sends it to the consumer and supplier. $k\leftarrow 0$.
      \item The consumer and supplier respectively determine 
      the demand $d(p^k)$ and the supply $s(p^k)$, and send them to the ISO.
      \item If there is a gap between the demand $d(p^k)$ and supply $s(p^k)$, 
      the ISO reassign a new price $p^{k+1}$ to manage the balance of supply and demand, 
      and send it to the consumer and supplier again. $k\leftarrow k+1$.
      \item Step 2 and 3 are repeated.
    \end{enumerate}
  \end{ALGO}
  \begin{figure}[h]
    \centering
    \includegraphics[width=.6\columnwidth]{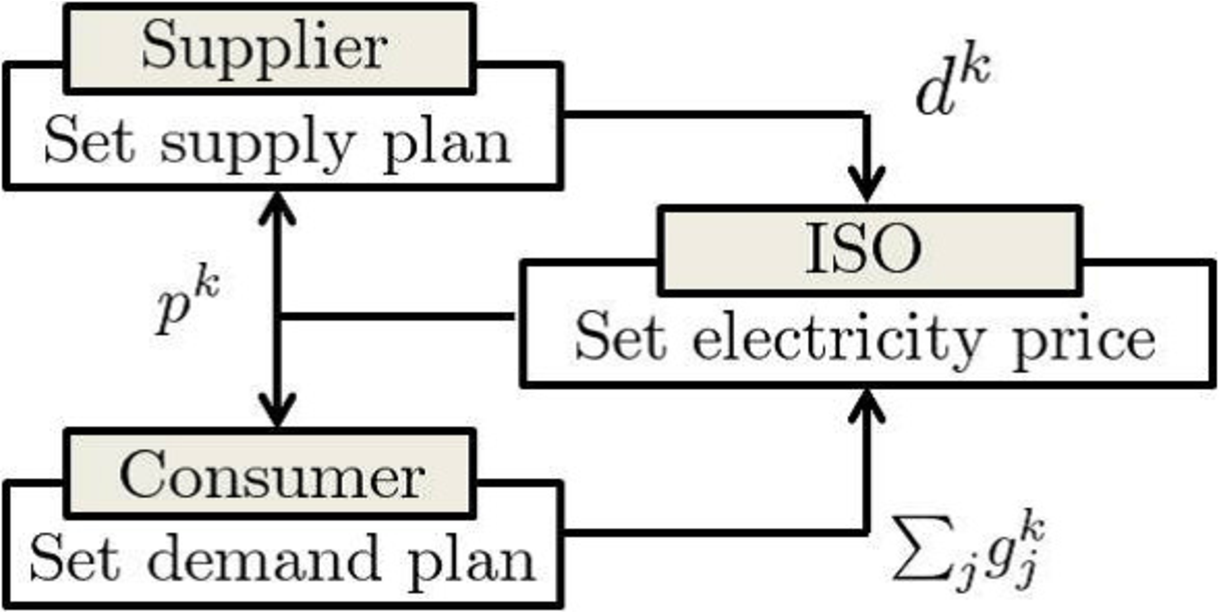}
    \caption{Schematic views of the electricity market model.}
  \end{figure}
  
  Miyano and Namerikawa \cite{Miyano2012a} proposed a pricing model 
  based on social welfare maximization problem 
  with the supply-demand balance constraint. 
  The social welfare is represented as the sum of the consumer's utility 
  and supplier's profit:
  \begin{align}
    \{u(d)-pd\} + \{py - v(y)\}.\label{social_welfare}
  \end{align}
  Since we assumed that the resistive losses can be ignored, 
  the supply-demand balance constraint is $d=y$.
  Consequently, we obtain the following social welfare maximization problem:
  \begin{align}
    \max_{d\geq0,y\geq0} u(d)-v(y) \quad \text{s.t.} \quad d=y.\label{Miyano_max}
  \end{align}
  Or more simply, 
  \begin{align}
    \max_{d\geq0} ~ u(d)-v(d).\label{Miyano_max_simple}
  \end{align}
  However, the ISO cannot solve \eqref{Miyano_max} 
  (or \eqref{Miyano_max_simple}) directly since the 
  utility function $u$ is unknown to the ISO. Therefore, 
  the following partial Lagrangian dual problem of \eqref{Miyano_max} 
  is alternatively considered in \cite{Miyano2012a}: 
  \begin{align}
    \min_{\lambda} ~ \varphi(\lambda),\label{Miyano_dual}
  \end{align}
  where
  \begin{align}
    \varphi(\lambda) &= \max_{d\geq0,y\geq0} u(d) - v(y) + \lambda(y - d)\nonumber\\
    &= \max_{d\geq0}\{u(d) -\lambda d\} + \max_{y\geq0}\{\lambda y - v(y)\}. \label{Miyano_varphi}
  \end{align}
  Since the maximization problems in \eqref{Miyano_varphi} correspond to 
  \eqref{consumer's_problem} and \eqref{supplier's_problem}, 
  the ISO can make the consumer and supplier solve them by 
  sending $\lambda$ as a price. This mechanism allows the ISO to 
  solve \eqref{Miyano_dual} by using the steepest descent method 
  (See \cite{Miyano2012a} for details). 
  The following lemma is well known (e.g. \cite{Miyano2012a,Roozbehani2010a}). 
  \begin{LEMM}\label{optimal_price}
    Suppose that $(d^*,y^*)$ is an optimal solution of \eqref{Miyano_max}, 
    and $\lambda^*$ is an optimal solution of \eqref{Miyano_dual}. 
    If Assumption \ref{convex_concave} is satisfied, $d^*=d(\lambda^*)$ 
    and $y^*=s(\lambda^*)$ hold. 
  \end{LEMM}
  $\lambda^*$ is called the Lagrange multiplier for \eqref{Miyano_max}. 
  Lemma \ref{optimal_price} states that the demand $d(\lambda^*)$ and 
  the supply $s(\lambda^*)$ maximize the social welfare \eqref{social_welfare}, 
  and therefore, the Lagrange multiplier $\lambda^*$ 
  can be regarded as an optimal price. 
  
  These results hold under the assumption that the cost function $v$ is convex. 
  Here, we consider an electricity market model with a nonconvex 
  cost function within the setting of the unit commitment 
  problem (UCP) and provide an efficient pricing model for it. 

\section{Convex hull pricing model}\label{Convex Hull Price}
  In this section, we will introduce 
  the convex hull pricing (CHP) model (a.k.a. extended locational marginal pricing model) 
  proposed by Gribik et al. \cite{Gribik2007}. 
  Note that the models in this section do not take into account dynamic demand; 
  We assume that the demand is given. 
  
  \subsection{Unit commitment problem}
  First, we consider the supplier's cost function. 
  \begin{figure}
    \centering
    \includegraphics[width=.35\columnwidth]{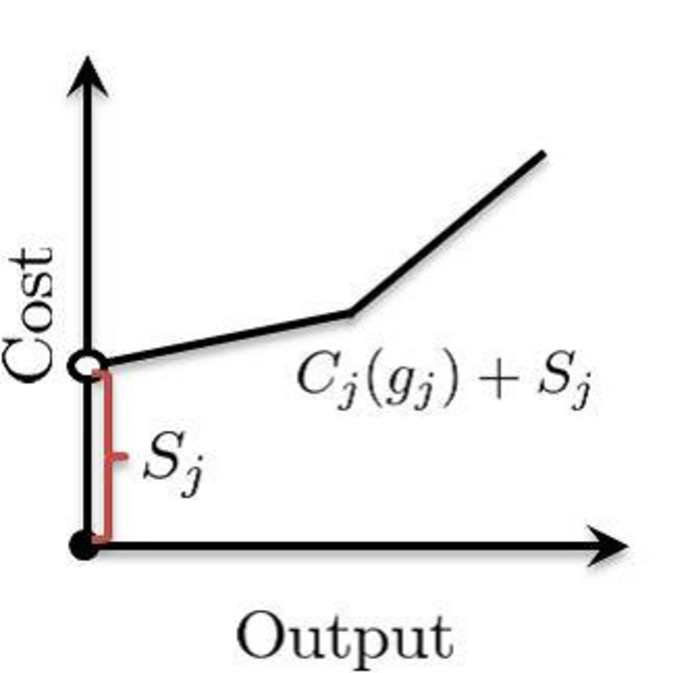}
    \caption{Example of a cost function of generator $j$. 
    $S_j$ denotes the startup cost and $C_j(g_j)$ denotes 
    the lower semicontinuous variable cost for an output $g_j$.}\label{generator}
  \end{figure}
  We assume that the supplier has several types of generators 
  that might have different variable cost functions and 
  startup costs as shown in Fig. \ref{generator}. 
  The supplier would like to minimize 
  the generating cost for a given demand $y$. 
  Then, the cost function $v:[0,\infty)\rightarrow [0,\infty]$ 
  for the supplier can be represented as the optimal value of  
  the following optimization problem:
  \begin{align}
  v(y):=\left\vert
    \begin{aligned}
       \inf_{\g,\z} ~& \sumj C_j(g_j) + \sumj S_j z_j \\
       \text{s.t. } ~ & \sumj g_j = y \\
                      & m_j z_j \leq g_j \leq M_j z_j,\quad z_j\in\{0,1\},~\forall j,
    \end{aligned}\right.\label{UCP}
  \end{align}
  where $C_j(g_j)$ is a lower semicontinuous 
  generating cost function for an output $g_j$; 
  $S_j$ denotes the startup cost; $m_j$ and $M_j$ denote the minimum and 
  maximum outputs; and $z_j$ represents a decision 
  to commit generator $j$. Problem \eqref{UCP} is called 
  a {\em unit commitment problem (UCP)}. 
  The optimal value function $v$ of the UCP is characterized by 
  lower semicontinuity, and it may be nonconvex and discontinuous. 
  
  \subsection{Convex hull price}
  A number of studies have proposed pricing models in the settings of UCP. 
  Some of them (e.g. \cite{Hogan2003,Gribik2007,Bjorndal2008,Zhang2009}) 
  have tried to reduce the {\em uplift payment} which is defined as follow. 
  \begin{DEFI}[Uplift payment]
    Suppose that a demand $y$ is given. 
    The uplift payment $\Pi(p;y)$ for a price $p$ is defined as follows:
    \begin{align} 
      \Pi(p;y) := \sup_x\{px-v(x)\} - \{py-v(y)\}.\label{uplift_eq}
    \end{align}
  \end{DEFI}
  \begin{figure}
    \centering
    \includegraphics[width=.35\columnwidth]{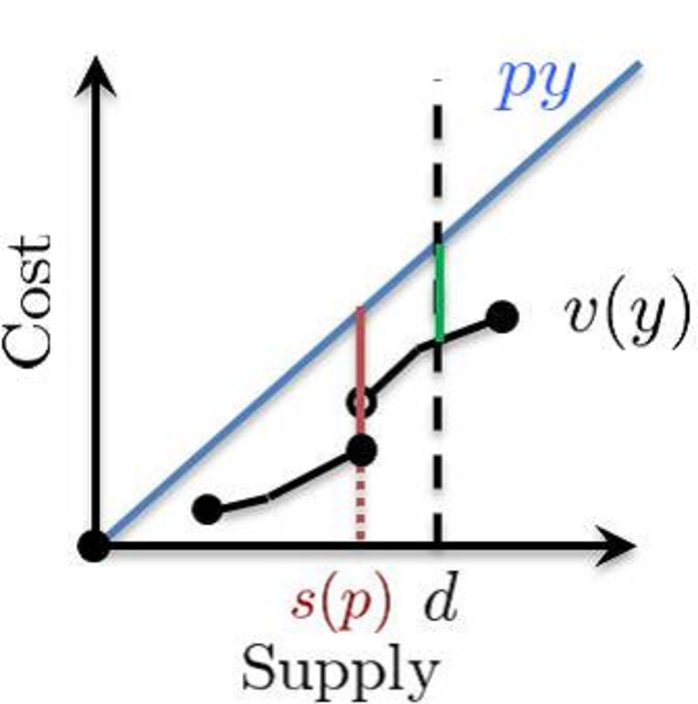}
    \caption{Illustration of the uplift payment. 
    In this picture, $py$ represents net sales. 
    The margin between $py$ and $v(y)$ represents the profit for the supplier. 
    In this case, the optimal profit for the supplier is obtained at $s(p)$.
    However, demand $d$ is given and the supplier has to supply $d$. 
    Thus, the actual profit for the supplier is obtained at $d$. 
    The gap between the optimal profit and the actual profit is called the uplift payment. 
    It is formulated as \eqref{uplift_eq}.}\label{uplift_fig}
  \end{figure}
  A conceptual illustration is given in Fig. \ref{uplift_fig}. 
  The first term of \eqref{uplift_eq} represents the maximum profit 
  for the supplier, that is realizable under a price $p$. 
  The second term of \eqref{uplift_eq} represents the actual profit 
  for the supplier under a price $p$ and the given demand $y$. 
  The uplift payment can be regarded as a measure of disadvantage for 
  the supplier, or as a cost for the ISO to incentive the supplier 
  to supply the given demand $y$ under the price $p$. 
  Therefore, it is natural to find the price that minimizes 
  the uplift payment. 
  
  Such a price can be obtained by the convex hull pricing (CHP) model 
  (a.k.a. extended locational marginal pricing model)
  proposed by Gribik et al. \cite{Gribik2007}. 
  Before introducing the CHP, we define the convex hull of the cost function $v(y)$:
  \begin{DEFI}[Convex hull of the cost function $v$]
    The convex hull $v^h$ of the cost function $v$ is defined as 
    \[v^h(y) := \inf_{\mu}\{\mu \mid (y,\mu)\in \text{conv}(\text{epi}(v))\},\]
    where $conv(A)$ is the convex hull of a set $A$, and $epi(v)$ denotes the epigraph of $v$.
  \end{DEFI}
  \begin{figure}
    \centering
    \includegraphics[width=.35\columnwidth]{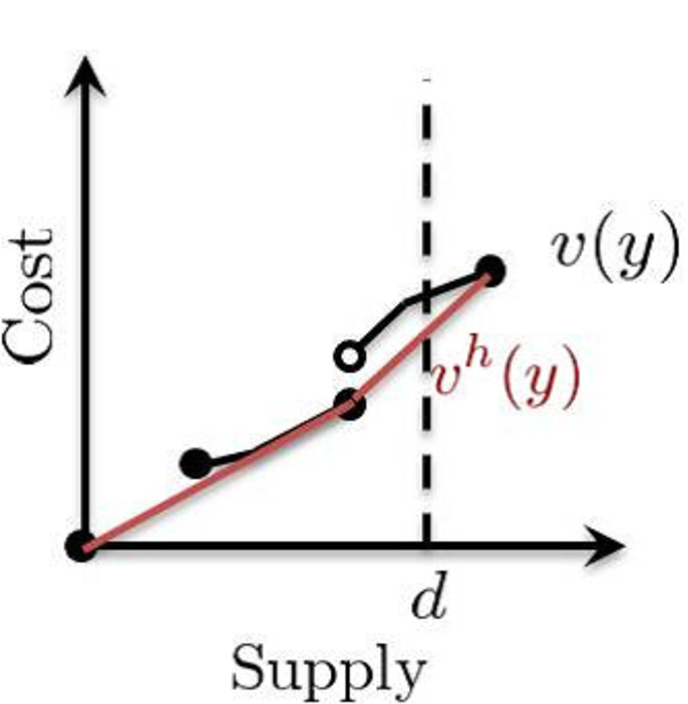}
    \caption{Illustration of the convex hull $v^h$ of the cost function $v$. 
    The CHP model uses the subgradient of $v^h$ at the given demand $d$ as a price $p^h$.}\label{CHcost}
  \end{figure}
  We illustrate the convex hull $v^h$ of the cost function $v$ in Fig. \ref{CHcost}. 
  The convex hull $v^h$ can be seen as the largest convex function 
  that is bounded above by $v$ at any point in its domain.  
  In addition, it is known that $v^h$ is coincide with biconjugate 
  (i.e., convex conjugate of convex conjugate) of $v$, if $v$ is lower semicontinuous 
  (see \cite{Bertsekas2003Convex}). Note that $v^h$ may be a nonsmooth function. 
  Now, we can define the CHP.
  \begin{DEFI}[Convex hull price]
    The convex hull price  $p^h$ is defined as the subgradient of 
    the convex hull $v^h$ of the cost function $v$ at a given demand $y$
    \footnote{Note that the convex hull price coincides with 
    the marginal cost price if $v$ is convex and differentiable.}, i.e., 
    \[ p^h \in \partial v^h(y).\]
  \end{DEFI}
  It is difficult to calculate the CHP according to the definition 
  because the explicit function form of $v(\cdot)$ and its convex 
  hull $v^h(\cdot)$ are generally too complicated to compute. 
  To address this issue, 
  Gribik et al. investigated the following important properties of the CHP 
  in connection with duality theory. 
  \begin{PROP}[\cite{Gribik2007}]\label{min_uplift}
    Suppose that $\lambda^*$ be an optimal solution of the following 
    partial Lagrangian dual problem of \eqref{UCP}:
    \begin{align}
      \max_{\lambda}\left\vert
      \begin{aligned}
        \min_{\g,\z} ~ & \sumj C_j(g_j) + \sumj S_jz_j + \lambda(y-\sumj g_j)\\
        \text{s.t. } ~ & m_j z_j \leq g_j \leq M_jz_j,\quad z_j\in\{0,1\},~ \forall j.
      \end{aligned}\right. \label{UCP_dual}
    \end{align}
    Then, $\lambda^*$ is a convex hull price, i.e., 
      \[\lambda^* \in \partial v^h(y).\]
    Moreover, $\lambda^*$ minimizes the uplift payment $\Pi(\cdot;y)$, i.e.,
      \[\lambda^* \in \argmin_{p\geq 0} ~ \Pi(p;y).\]
  \end{PROP}
  Proposition \ref{min_uplift} states that the Lagrangian multiplier 
  for \eqref{UCP} gives the CHP, and the CHP is the best price 
  in the sense of reducing the uplift payment $\Pi(\cdot;y)$. 
  Many researcher have studied algorithms for 
  \eqref{UCP_dual}, e.g., \cite{Wang2013c,Wang2010}. 
  However, there still remains a difficulty to solve 
  a mixed integer programming problem in \eqref{UCP_dual}. 
\section{Pricing Model and Algorithm}\label{Model and Algorithm}
  In this section, we present a new dynamic pricing model 
  and algorithm that deals with the nonconvex cost function $v$ 
  of \eqref{UCP}. Here, we assume that the consumer's utility function $u$ 
  is concave, non-decreasing, $\lim_{d\rightarrow 0}\dot{u}(d)=0$. 
  
 \subsection{Pricing model}
  By adding $v$ of \eqref{UCP} to \eqref{Miyano_max}, 
  we derive a new social welfare maximization problem 
  as follows:
  \begin{align}
    \max_{\g,\z,d}~& u(d) - \left\{\sumj C_j(g_j) + \sumj S_jz_j\right\}\nonumber\\
    \text{s.t. }~& \sumj g_j=d, \quad d \geq 0\label{ISO_opt_proposed}\\
                 & z_jm_j\leq g_j \leq z_jM_j,\quad z_j\in\{0,1\}.\quad \forall j,\nonumber
  \end{align}
  Since the ISO does not know the utility function $u$, 
  it cannot solve \eqref{ISO_opt_proposed}. Therefore, we alternatively 
  consider a partial Lagrangian dual problem of \eqref{ISO_opt_proposed}.
  For notational simplicity, 
  let $X$ denotes the feasible set for outputs $\g$ 
  and commitment decisions $\z$, i.e.,
  \begin{align*}
    X:=\{(\g,\z)~\vert~ m_jz_j\leq g_j \leq M_jz_j, ~z_j\in\{0,1\},\forall j\}.
  \end{align*}
  The partial Lagrangian dual problem of \eqref{ISO_opt_proposed} 
  is formulated by adding $v$ of the UCP \eqref{UCP} to \eqref{Miyano_dual} as follows: 
  \begin{align}
    \min_{\lambda} ~ \varphi(\lambda), \label{ISO_dual_proposed}
  \end{align}
  where
  \begin{align}
    \varphi(\lambda) :&=\max_{d\geq 0} \{u(d) - \lambda d\}\nonumber\\
    &~~~+\max_{(\g,\z)\in X}\left\{\lambda\sumj g_j - {\left\{\sumj C_j(g_j) + \sumj S_jz_j\right\}}\right\}. \label{dual_function}
  \end{align}
  The objective function $\varphi$ may be nonsmooth 
  because of the nonsmoothness of the cost function $v$. 
  Note that the problems \eqref{ISO_opt_proposed} and \eqref{ISO_dual_proposed} 
  take into account price-sensitive demands. 
  Hence they are different from existing CHP models such as those 
  in \cite{Gribik2007,Wang2013c,Wang2010}. 
  Now we have reached the following result. 
  \begin{PROP}\label{our_CHP} 
  The following statements hold:
    \begin{enumerate}
      \item The partial Lagrangian dual problem \eqref{ISO_dual_proposed} has 
      an optimal solution $(\lambda^*,\g^*,\z^*,d^*)$. \label{first_part}
      \item The problem \eqref{ISO_dual_proposed} is equivalent to 
      \[ \max_{d\geq 0} ~ u(d) - v^h(d), \]
      \item $\lambda^*$ is a convex hull price, i.e., \label{third_part}
      \[ \lambda^* \in \partial v^h(d^*). \]
    \end{enumerate}
  \end{PROP}
  \begin{proof} 
  The idea behind the proof of \ref{first_part}) is to use 
the general version of Weierstrass' Theorem \cite{Bertsekas2003Convex}, 
which states that $\varphi$ has a minimum point if $\varphi$ is 
a closed\footnote{A function $f$ is said to be {\em closed} 
if epi$(f)$ is a closed set.} 
proper\footnote{A function $f$ is said to be {\em proper} 
if there exists $\x\in\mathbb{R}^n$ such that $f(\x)\not = \infty$ 
and there does not exist $\x '\in\mathbb{R}^n$ such that $f(\x ')=-\infty$.} 
function and has a nonempty and bounded level set. 
From duality theory it is known that $\varphi$ is lower semicontinuous, 
and this guarantees the closedness of $\varphi$. 
For all $\lambda < 0$, $\varphi(\lambda)=\infty$ holds since 
the first term in \eqref{dual_function} is infinite, $\infty$. 
For all $\lambda \geq 0$, the first term in \eqref{dual_function} 
is finite, and an optimal solution exists. 
The second term in \eqref{dual_function} is also finite, and 
an optimal solution exists because of the compactness of $X$.
Thus, if $\lambda^*$ exists, $\lambda^*$ is a nonnegative number and 
$(\g^*,\z^*,d^*)$ also exists  
(This implies that $d\leq\sum_jg_j$ is an effective constraint). 
Let us choose $\gamma\in \mathbb{R}$ so that the level set 
$L=\{\lambda ~\vert~ \varphi(\lambda)\leq \gamma\}\in (0,\infty)$ is non-empty. 
From \eqref{dual_function}, we have
\[\varphi(\lambda)\geq u(0) + \lambda\sumj M_j - (\sumj C_j(M_j) + \sumj S_j)\]
for any $\lambda$. The right-hand side of the inequality is derived by setting 
$d=0$, $g_j=M_j$, and $z_j=1$ to \eqref{dual_function}. 
There exists a sufficiently large $\hat{\lambda}>0$ 
such that $\varphi(\hat{\lambda})>\gamma$. 
Hence, $L$ is bounded, 
and the conditions on the existence of a minimum point of 
$\varphi$ are satisfied. 

The remaining parts can be 
obtained by mimicking the argument in \cite{Gribik2007}. 
The second term in \eqref{dual_function} is written as 
\begin{align*}
\max_{\eta,(\g,\z)\in X}\left\{\lambda\eta - {\left\{\sumj C_j(g_j)+\sumj S_jz_j\right\}}~\vert~\eta=\sumj g_j\right\}\\
=\max_{\eta}\{\lambda \eta - v(\eta)\} = v^c(\lambda),
\end{align*}
where $v^c$ is the convex conjugate function of $v$. 
Problem \eqref{ISO_dual_proposed} can be expressed as 
\begin{align*}
\min_{\lambda} \varphi(\lambda)&=\min_{\lambda}\max_{d\geq 0}\{u(d)-\lambda d +v^c(\lambda)\}\\
&=\max_{d\geq 0}\min_{\lambda}\{u(d)-\lambda d +v^c(\lambda)\}\\
&=\max_{d\geq 0}\{u(d) - v^{cc}(d)\},
\end{align*}
where $v^{cc}$ is the biconjugate function of $v$, i.e., 
the convex conjugate function of $v^c$. The second equality 
holds from \cite[Prop. 2.6.4]{Bertsekas2003Convex}. 
To prove \ref{third_part}), we obtain the following result for all $d\geq 0$:
\begin{align*}
v^{cc}(d) &= \sup_{\lambda}\{d\lambda-v^c(\lambda)\}\\
&\geq d\lambda^* - v^c(\lambda^*)\\
&= d\lambda^* -v^c(\lambda^*)+\lambda^*d^*-\lambda^*d^*\\
&= v^{cc}(d^*)+\lambda^*(d-d^*).
\end{align*}
This implies that $\lambda^*$ is a subgradient of $v^{cc}$ at $d^*$. 
Since $v^{cc}=v^h$ holds (see \cite{Bertsekas2003Convex}), 
this completes the proof.
  \end{proof}
  Proposition \ref{our_CHP} states that the Lagrange multiplier 
  $\lambda^*$ for the social welfare maximization problem \eqref{ISO_opt_proposed} gives a CHP. 
  This implies that $\lambda^*$ minimizes the uplift payment 
  at the equilibrium demand $d^*$. Thus, it is reasonable 
  to set $\lambda^*$ as a price.

  \subsection{Subgradient algorithm}
  The fact that the dual function $\varphi$ in \eqref{ISO_dual_proposed} 
  is convex and lower semicontinuous is known from duality theory. 
  Using a basic optimization method, the dual problem 
  \eqref{ISO_dual_proposed} can be solved in the electricity market model. 
  Here we provide the following subgradient algorithm for \eqref{ISO_dual_proposed}. 
  \begin{ALGO} \label{subgradient}
    Set the step size $\gamma^k>0$ $(k=1,2,...,N)$.
    \begin{enumerate}
      \item The ISO sets the initial price $\lambda^0$. $k \leftarrow 1$.
      \item According to the given price $\lambda^k$, 
      the consumer and supplier adjust their respective 
      demand $d^k$ and supply $\sumj g_j^k$ by solving 
      the following utility and profit maximization problems:
      \begin{align*}
        d^k &\in \argmax_{d\geq 0} ~ u(d) - \lambda^k d,\\
        (\g^k,\z^k) &\in \\
        &\argmax_{(\g,\z)\in X} ~ \lambda^k\sumj g_j - \{\sumj C_j(g_j) + \sumj S_jz_j\}.
      \end{align*}
      \item The ISO updates the price $\lambda^{k+1}$ as 
      \[ \lambda^{k+1} = \lambda^k - \gamma^k(\sumj g_j^k - d^k). \]
      $k \leftarrow k+1$. 
      \item Repeat steps 2 and 3 $N$ times. 
      The ISO then accepts the price $\lambda^N$. 
      The scheduled levels of 
      production and load are both $d^N$. 
    \end{enumerate}
  \end{ALGO}
  It is known that $\sumj g_j^k - d^k$ is a subgradient of $\varphi$ at $\lambda^k$. 
  The subgradient algorithm for convex optimization with appropriate step sizes 
  is proven to converge to a minimum (see \cite{Bertsekas2003Convex}). 
  Note that Algorithm \ref{subgradient} is an approximation algorithm 
  because the stopping criterion is defined by the number of iterations. 
  In the electricity market model that we assumed, 
  the ISO has to send the information to the supplier and the consumer 
  at each iteration, and vice versa. 
  Therefore, it is unrealistic to iterate many times. 
  Moreover, although the necessary and sufficient optimality conditions for \eqref{ISO_dual_proposed} can be obtained by modifying \cite[eq.(9)]{Wang2013c}, we have to solve an additional convex quadratic program \cite[eq.(12)]{Wang2013c} at each iteration in order to check the conditions. 
  These calculation may cost high. 
  The numerical results in the next section show that Algorithm \ref{subgradient} reduces 
  the uplift payment enough small in a few iterations. Thus 
  it would not be a big disadvantage to define the stopping criterion 
  by the number of iteration. 
\section{Numerical Experiments}\label{Numerical Simulations}
  Here, we present numerical experiments that confirm the 
  efficiency of our pricing model and algorithm in reducing 
  uplift payments. We assume that the ISO makes the next day's 
  pricing decision hourly, i.e., a day-ahead market. 
  We did the numerical experiments on 
  a Intel Core i5 M540 processor (2.53 GHz) and 
  4GB of physical memory 
  with Windows 7 Professional 64bit Service Pack 1. 
  Numerical algorithms were written in R language 
  version 3.0.1. The GNU Linear Programming Kit package of 
  version 0.3-10 was used for solving linear programming 
  and mixed integer programming problems. 

  \subsection{Cost functions}
    First, we shall consider cost functions of the supplier. 
    We used the examples from Gribik et al. \cite{Gribik2007} and Hogan et al. \cite{Hogan2003} 
    (shown in TABLEs \ref{Gribik_example} and \ref{Scarf_example}). 
    The example from \cite{Hogan2003} is the modification of Scarf example \cite{Scarf1994}. 
    The generators of the examples have piece-wise linear variable cost functions. 
    \begin{table}
      \centering
      \caption{Gribik Example \cite{Gribik2007}}\label{Gribik_example}
      \begin{tabular}{|c|cccccc|} \hline
      \multicolumn{1}{|c}{} & \multicolumn{6}{c|}{Generators} \\
      \multicolumn{1}{|c}{} & \multicolumn{2}{c}{A} & \multicolumn{2}{c}{B} & \multicolumn{2}{c|}{C} \\\cline{2-7}
      \multicolumn{1}{|c|}{} & Var1 & Var2 & Var1 & Var2 & Var1 & Var2 \\
      \multicolumn{1}{|c|}{Capacity (MW)} & 100 & 100 & 100 & 100 & 100 & 100 \\
      Minimum output (MW) & \multicolumn{2}{c}{0} & \multicolumn{2}{c}{0} & \multicolumn{2}{c|}{0} \\
      Startup cost (\$) & \multicolumn{2}{c}{0} & \multicolumn{2}{c}{6000} & \multicolumn{2}{c|}{8000} \\
      Var cost 1 (\$/MW)& \multicolumn{2}{c}{65} & \multicolumn{2}{c}{40} & \multicolumn{2}{c|}{25} \\
      Var cost 2 (\$/MW)& \multicolumn{2}{c}{110} & \multicolumn{2}{c}{90} & \multicolumn{2}{c|}{35} \\
      Number of units & \multicolumn{2}{c}{1} & \multicolumn{2}{c}{1} & \multicolumn{2}{c|}{1} \\\hline
      \end{tabular}
    \end{table}
    \begin{table}
      \centering
      \caption{Modified Scarf Example \cite{Hogan2003}}\label{Scarf_example}
      \begin{tabular}{|c|ccc|} \hline
      \multicolumn{1}{|c}{} & \multicolumn{3}{c|}{Generators} \\
      \multicolumn{1}{|c}{} & \multicolumn{1}{c}{Smokestack} & High Tech & MedTech \\\cline{2-4}
      Capacity & 16 & 7 & 6 \\
      Minimum Output & 0 & 0 & 2 \\
      Startup Cost & 53 & 30 & 0 \\
      Var cost & 3 & 2 & 7 \\
      Number of Units & 6 & 5 & 5 \\\hline
      \end{tabular}
    \end{table}
    \begin{figure}
      \begin{minipage}{.49\columnwidth}
        \centering
        \includegraphics[width=\columnwidth,bb=0 23 475 480,clip]{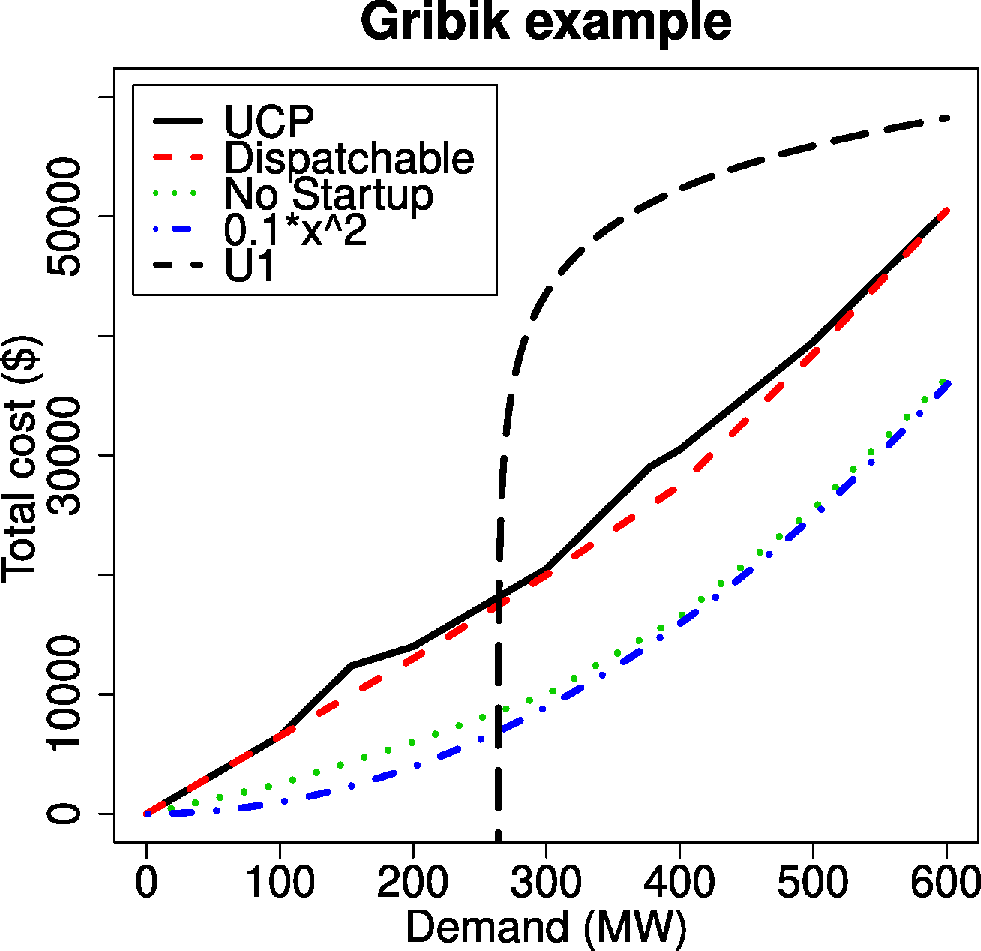}
      \end{minipage}
      \begin{minipage}{.49\columnwidth}
        \centering
        \includegraphics[width=\columnwidth,bb=0 23 475 480,clip]{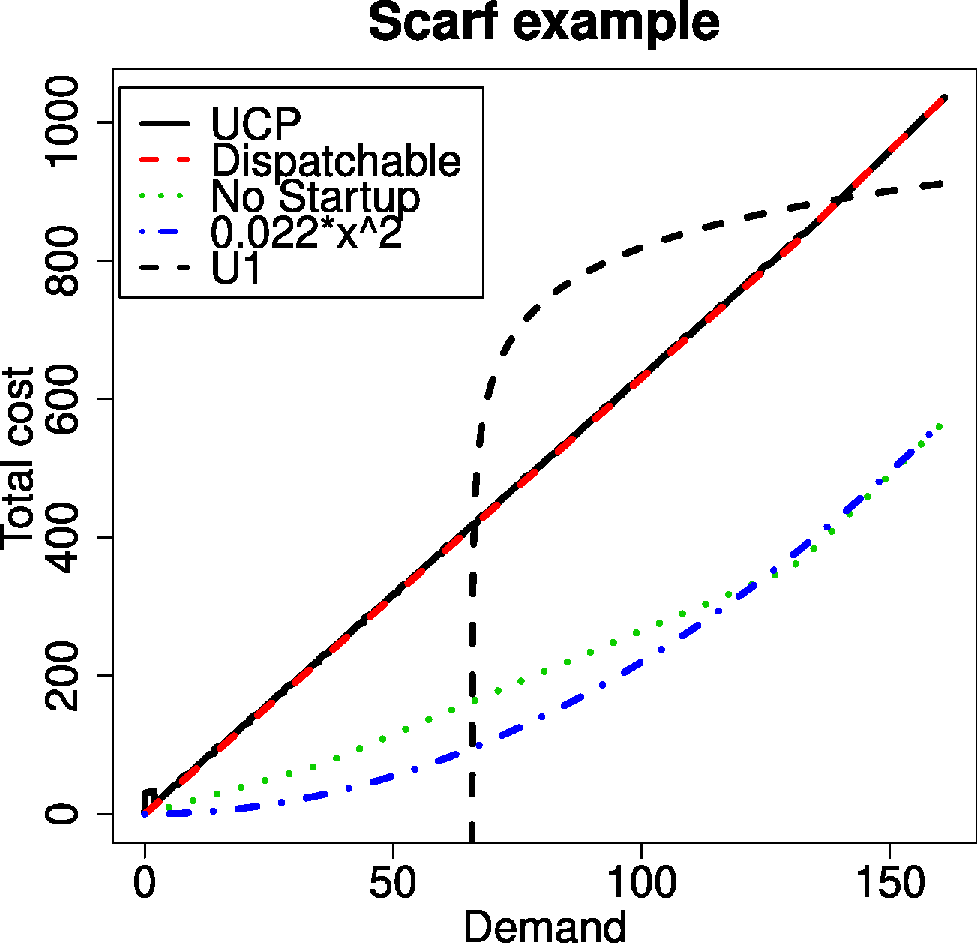}
      \end{minipage}
      \caption{Cost and utility functions}\label{CU}
    \end{figure}
    The four different cost functions with the two examples 
    are illustrated in Fig. \ref{CU}. 
    `UCP' means the optimal value function $v(y)$ of the UCP \eqref{UCP}, i.e., 
    the actual cost function for the supplier. `Dispatchable' means the 
    optimal value function of the continuous relaxation model of \eqref{UCP}; 
    i.e., the model replaces the 0-1 integer constraints $z_j\in\{0,1\}$ in 
    \eqref{UCP} by $0\leq z_j \leq 1$ for all $j$. 
    The dispatchable pricing model in \cite{Gribik2007} 
    uses a subgradient of `Dispatchable' as a price. 
    We used this model for comparison of the uplift payments. 
    `No Startup' means the optimal value function of \eqref{UCP} 
    that ignores the startup cost of the generators, i.e., 
    \eqref{UCP} with $S_j=0$ for all $j$. 
    We also used a quadratic cost function that approximates `No Startup' 
    and that satisfies the assumption of 
    existing convex cost models \cite{Roozbehani2011,Miyano2012a} 
    (i.e., Assumption \ref{convex_concave}). 

    We can use a smooth and convex function approximating `UCP' in the existing models. 
    However, the use of such function is unrealistic 
    since the explicit function form of `UCP' is generally 
    too complicated to compute as we noted in Section \ref{Convex Hull Price}. 
    Note that our model doesn't need to calculate `UCP', 
    although our model sets a subgradient of the convex hull of `UCP' as a price.

\subsection{Demand function}
  Following \cite{Miyano2012a}, we used the hourly demand function, 
  \begin{align}
    D_t(\lambda) = \mu_1\nu d_{1,t} + \mu_2(1+\delta_{2,t})d(\lambda),
    \label{demand_function}
  \end{align}
  where $\mu_1,\mu_2,\nu$ are positive parameters, 
  $d_{1,t}$ is a positive constant, 
  and $\delta_{2,t}$ is a random variable 
  distributed with $\mathcal{N}(0,0.01^2)$ $(t=1,2,\dots,24)$. 
  The first term in \eqref{demand_function} represents the minimum necessary demand, 
  and the second term in \eqref{demand_function} represents the swing in demand 
  depending on prices. 
  We used actual hourly demand data of 
  the Tokyo Electric Power Company from 0:00 to 23:00 
  August 30, 2012 \cite{TEPCO} as $d_{1,t}$ $(t=1,2,...,24)$. 
  $\nu$ is a scale parameter to adapt the demand to the examples' size. 
  We defined $d(\lambda)$ as a solution of the following 
  utility maximization problem: 
  \[ d(\lambda):=\argmax_{d\geq0}\{u(d)-\lambda d\},\]
  where $u(d)$ is a logarithmic utility function $a\log(d)$ 
  with a positive parameter $a$. 
  Positive parameters $a$, $\mu_1$, $\mu_2$ are adjusted 
  so that the sum of simulated hourly demands 
  $D_t(\lambda_t^N)$ remained nearly equal to the sum of 
  the scaled actual hourly demands $\nu d_{1,t}$, i.e.,
  \[ \sum\nolimits_{t=1}^{24} D_t(\lambda_t^N) \approx \nu\sum\nolimits_{t=1}^{24} d_{1,t}. \]
  Parameter settings for the demand function \eqref{demand_function} 
  are shown in the TABLE \ref{Params}. 
  \begin{table}
    \centering
    \caption{Parameter Settings for demand function \eqref{demand_function}}\label{Params}
    \begin{tabular}{|c|c|c|c|c|c|} \hline
     & $N$ & $a$ & $\nu$ & $\mu_1$ & $\mu_2$ \\\hline
    Gribik & 100 & $3.9\times10^4$ & 0.01 & 0.8 & 0.2 \\\hline
    Scarf & 100 & 455 & 0.0025 & 0.8 & 0.2 \\\hline
    \end{tabular}
  \end{table}
  \begin{table}
    \centering
    \caption{Summary of $d_{1,t}$}\label{Summary}
    \begin{tabular}{|c|c|c|c|} \hline
     & min & mean & max \\\hline
    $d_{1,t}$ & 28340.0 & 41086.7 & 50780.0 \\\hline
    \end{tabular}
  \end{table}

\subsection{Utility function}
  In Section \ref{Market Model}, we assumed that the consumer 
  determines their demand $D_t$ to maximize their utility, i.e.,
  \[D_t(\lambda) = \argmax_{D \geq 0}\{U_t(D)-\lambda D\},\] 
  where $U_t$ is a utility function. In the case that the demand function 
  $D_t$ is given by \eqref{demand_function}, 
  $U_t$ can be represented as follows: 
  \[U_t(D) = a\mu_2(1+\delta_{2,t})\log(D-\mu_1\nu d_1) + C,\]
  where $C$ is a constant. 
  The concave lines in Fig. \ref{CU} 
  illustrate $U_1(D)$ with $C=20000$ for Gribik example 
  and $C=500$ for Scarf example. 
  The margin between $U_1(D)$ and `UCP' represents the social welfare \eqref{social_welfare} 
  under the condition that the supply and demand are equal. 

\subsection{Comparison of pricing models}
  \subsubsection{Uplift payments}
  First, let us compare the uplift payments of our model 
  with the convex cost model in \cite{Roozbehani2011, Miyano2012a}. 
  The convex cost model uses the quadratic functions in Fig. \ref{CU} as cost functions. 
  Its electricity price is given by a locational marginal price (LMP) 
  (see \cite{Roozbehani2011, Miyano2012a} for details). 
  On the other hand, our electricity price is given by the CHP. 
  The hourly optimal prices $\lambda^*_t~(t=1,2,\dots,24)$ for our model were 
  calculated with \eqref{ISO_dual_proposed},  
  whereas optimal prices for the existing convex cost model 
  were calculated with \cite[eq.(7)]{Miyano2012a}. 
  \begin{figure}
    \begin{minipage}{.49\columnwidth}
      \centering
      \includegraphics[width=\columnwidth,bb=0 23 475 480,clip]{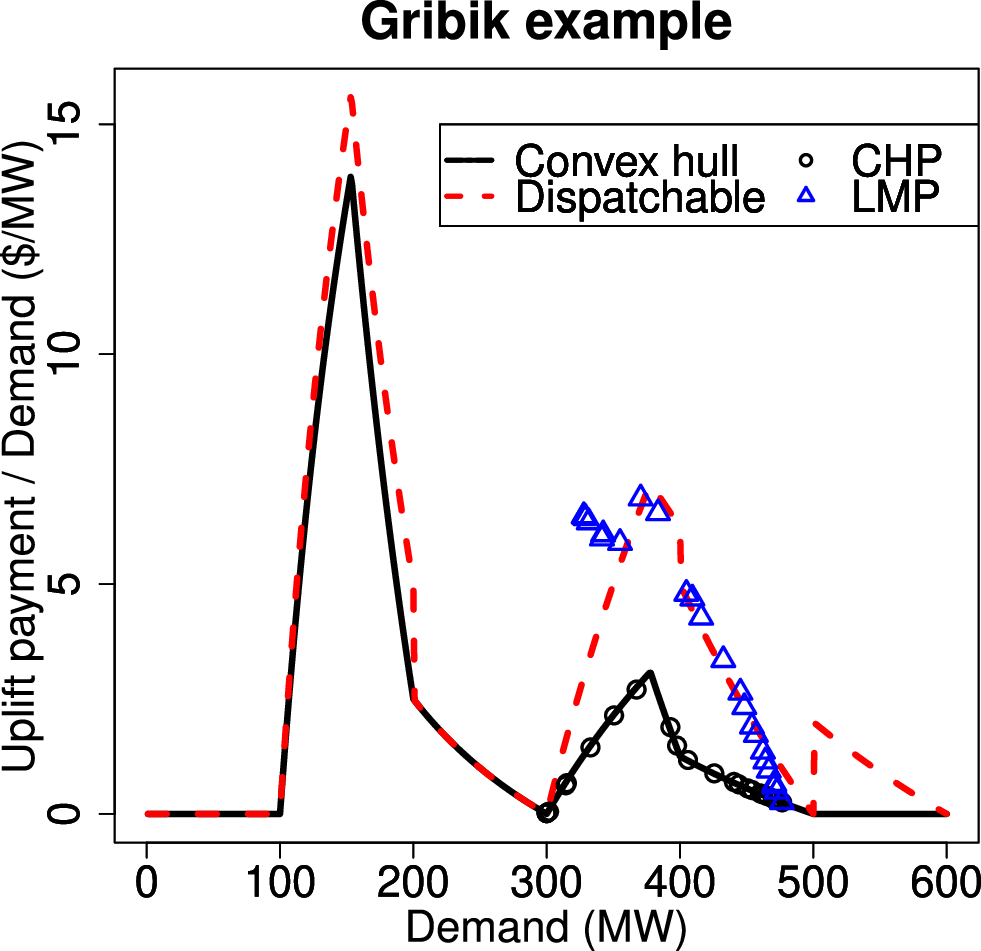}
    \end{minipage}
    \begin{minipage}{.49\columnwidth}
      \centering
      \includegraphics[width=\columnwidth,bb=0 23 475 480,clip]{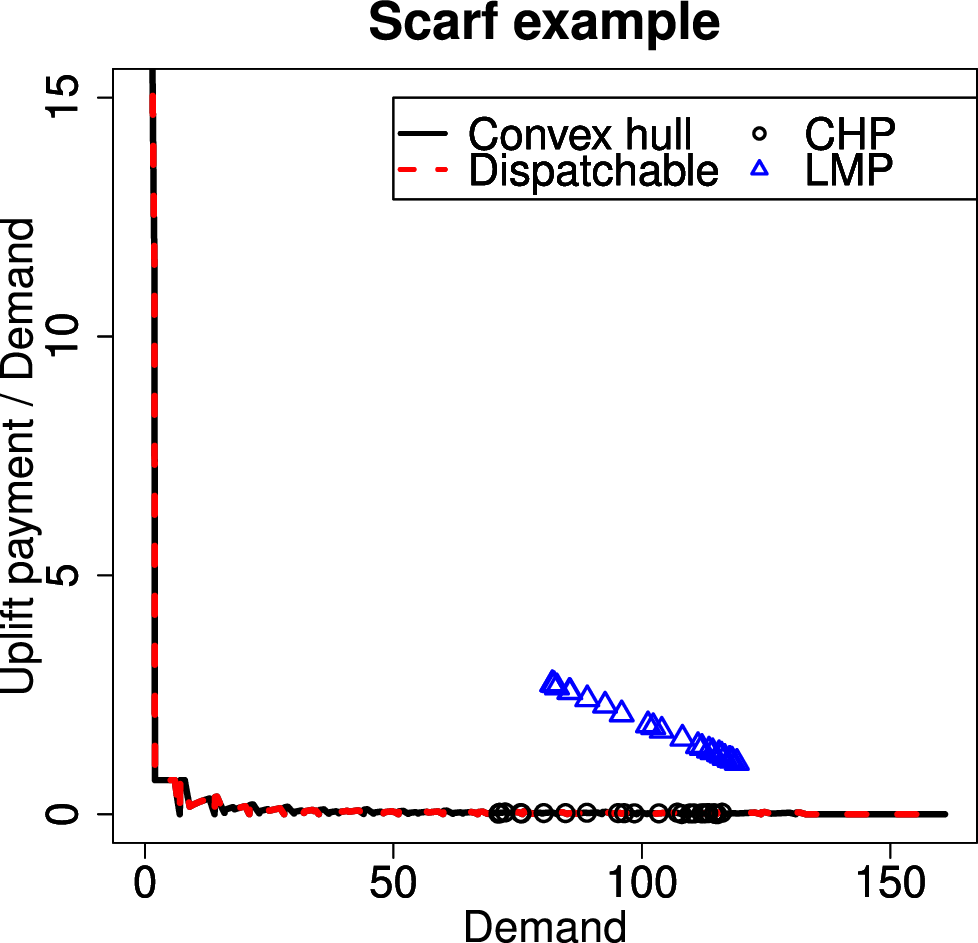}
    \end{minipage}
    \caption{Comparison of uplift payments of CHPs and LMPs}\label{CHP_LMP}
  \end{figure}
  Each point in Fig. \ref{CHP_LMP} shows the relation 
  between the equilibrium demand $D_t(\lambda_t^*)$ and 
  the uplift payments $\Pi(\lambda_t^*,D_t(\lambda_t^*))$ 
  for the hourly optimal prices $\lambda^*_t~(t=1,2,\dots,24)$. 
  The solid (dashed) line illustrates the uplift payments with the CHP 
  (dispatchable pricing) model in \cite{Gribik2007} at each fixed demand. 
  Note that the solid line shows the theoretical lower bound 
  of the uplift payments at each demand. 
  In our model, the uplift payments accrue less than half 
  that of the convex cost model. 
  
  \subsubsection{Demand, utility, and profit}
  \begin{figure}
    \begin{minipage}{.49\columnwidth}
      \centering
      \includegraphics[width=\columnwidth,bb=0 23 475 480,clip]{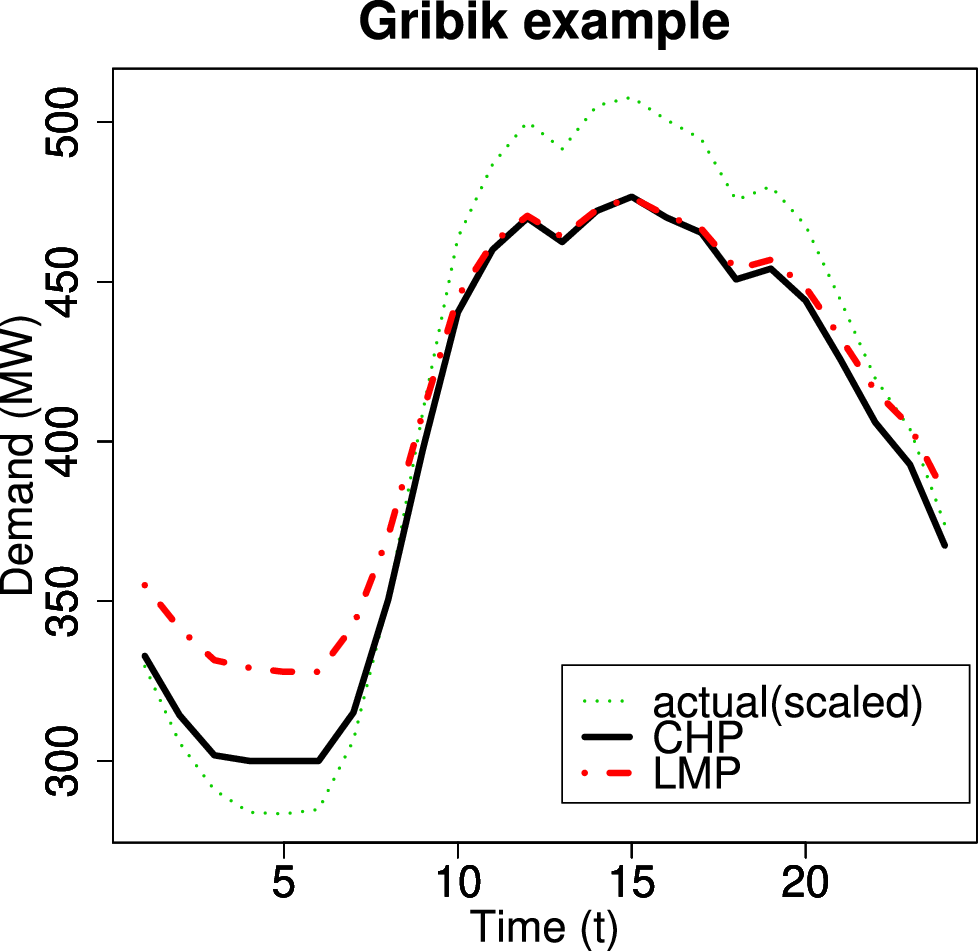}
    \end{minipage}
    \begin{minipage}{.49\columnwidth}
      \centering
      \includegraphics[width=\columnwidth,bb=0 23 475 480,clip]{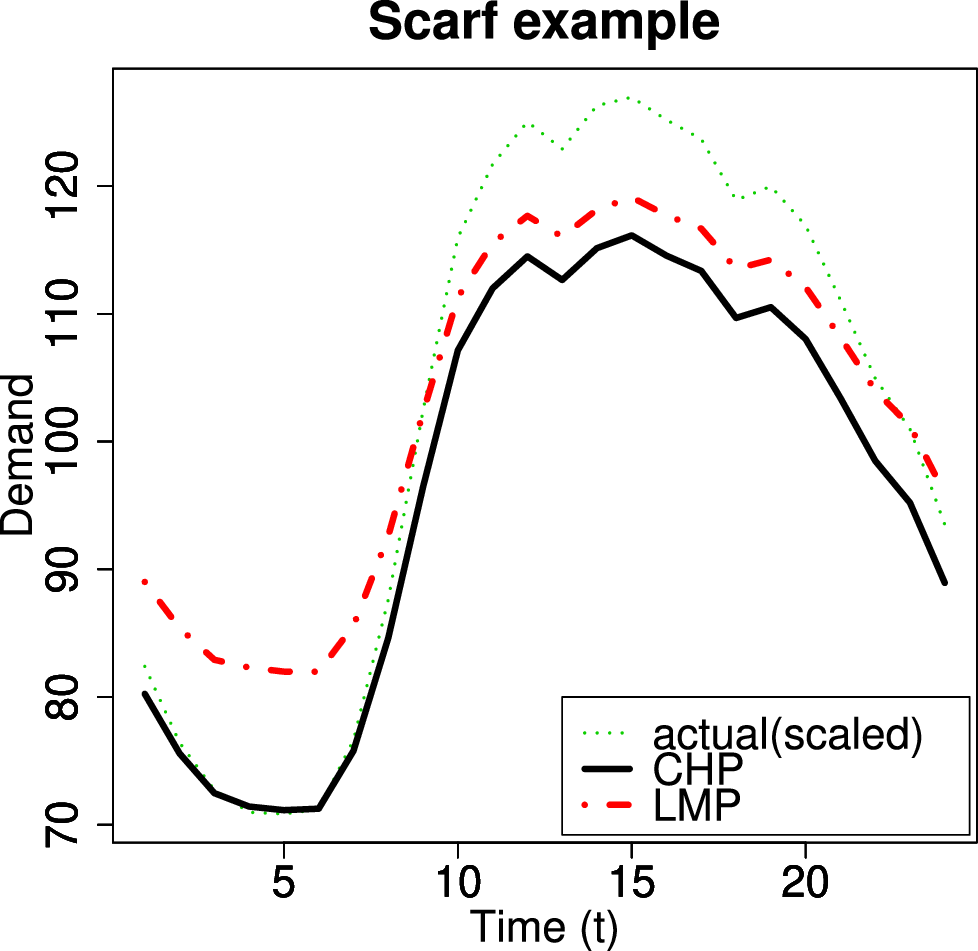}
    \end{minipage}
    \caption{Hourly equilibrium demand.}\label{demand}
  \end{figure}
  Fig. \ref{demand} 
  illustrates 
  the hourly equilibrium demand $D_t(\lambda^*_t)$. 
  In each example, our model tends to lead 
  less demand than that of the convex cost model, especially in the early morning. 
  It is because our model tends to set higher prices 
  than that of the existing model due to the startup costs (as shown in TABLE \ref{prices}); 
  Especially at an early hour, the startup costs occupy largely in total cost. 
  \begin{table}
    \centering
    \caption{Summary of simulated prices}\label{prices}
    \begin{tabular}{|c|c|c|c|c|} \hline
    \multicolumn{2}{|c|}{} & min & mean & max \\\hline
    \multicolumn{1}{|c|}{Gribik} & CHP & 87.1  & 94.0  & 95.2  \\\hline
    \multicolumn{1}{|c|}{} & LMP & 65.6  & 82.2  & 95.3  \\\hline \hline
    \multicolumn{1}{|c|}{Scarf} & CHP & 6.3  & 6.3  & 6.4  \\\hline
    \multicolumn{1}{|c|}{} & LMP & 3.6  & 4.5  & 5.2  \\\hline
    \end{tabular}
  \end{table}
  \begin{table}
    \centering
    \caption{Sum of simulated results with Gribik example}\label{sum_Gribik}
    \begin{tabular}{|c|c|c|c|} \hline
      & $\nu\sum d_{1,t}$   &  LMP   &  CHP   \\\hline
    Demand   &  9860.8   &  9860.4   &  9570.7   \\\hline
    Consumer's utility   & NA  &  354572.9   &  249772.4   \\\hline
    Supplier's profit   &  NA &  70362.1   &  181574.5   \\\hline
    Social welfare   &  NA &  424935.0   &  431346.9   \\\hline
    \end{tabular}
  \end{table}
  \begin{table}
    \centering
    \caption{Sum of simulated results with Scarf example}\label{sum_Scarf}
    \begin{tabular}{|c|c|c|c|} \hline
      & $\nu\sum d_{1,t}$   &  LMP   &  CHP   \\\hline
    Demand   &  2465.2   &  2465.368   &  2318.659   \\\hline
    Consumer's utility   & NA  &  7248.106   &  3203.201   \\\hline
    Supplier's profit   & NA  &  -4253.82   &  -31.5361   \\\hline
    Social welfare   & NA  &  2375.12   &  1802.465   \\\hline
    \end{tabular}
  \end{table}
  The sum of hourly results are summarized in TABLE \ref{sum_Gribik} and \ref{sum_Scarf}. 
  Since our model tends to set the price higher, 
  demand and consumer's utility would be lower 
  and supplier's profit would be higher. 
  On Scarf example, the supplier's profit takes negative value. 
  Because it is difficult to reap profit under 
  the price using a subgradient (or the gradient) 
  in the case that a cost function is almost linear 
  (even in such a case, 
  our model achieves more preferable results for the supplier than 
  the convex cost model). 
  If the ISO sets higher prices than LMPs and CHPs, 
  the supplier may make positive profit. However, such prices would 
  increase the uplift payments. 
  
  \subsubsection{Social welfare}
  While the social welfare of our model is larger than 
  that of the convex cost model for Gribik example, 
  it comes out opposite results for Scarf example 
  (see TABLEs \ref{sum_Gribik} and \ref{sum_Scarf}). 
  Although our model minimizes the uplift payments, 
  we cannot say that our model leads larger social welfare. 
  
\subsection{Comparison of pricing algorithms}
  Next, let us investigate the performance of pricing algorithms. 
  The algorithm based on the steepest descent method \cite{Miyano2012a} 
  was used for the convex cost model. 
  We choose the step sizes, which are used in Algorithm 1 in \cite{Miyano2012a} and 
  Algorithm \ref{subgradient} in this paper, 
  as $\gamma^k=1/10k$ for Gribik example 
  and as $\gamma^k=1/100k$ for Scarf example $(k=1,2,\dots,N)$. 
  The initial price was $\lambda^0_t=100$ for Gribik example 
  and $\lambda^0_t=10$ for Scarf example $(t=1,2,\dots,24)$. 
  
  \subsubsection{Change of uplift payments with respect to iterations.}
  \begin{figure}
    \begin{minipage}{.49\columnwidth}
      \centering
      \includegraphics[width=\columnwidth,bb=0 23 475 480,clip]{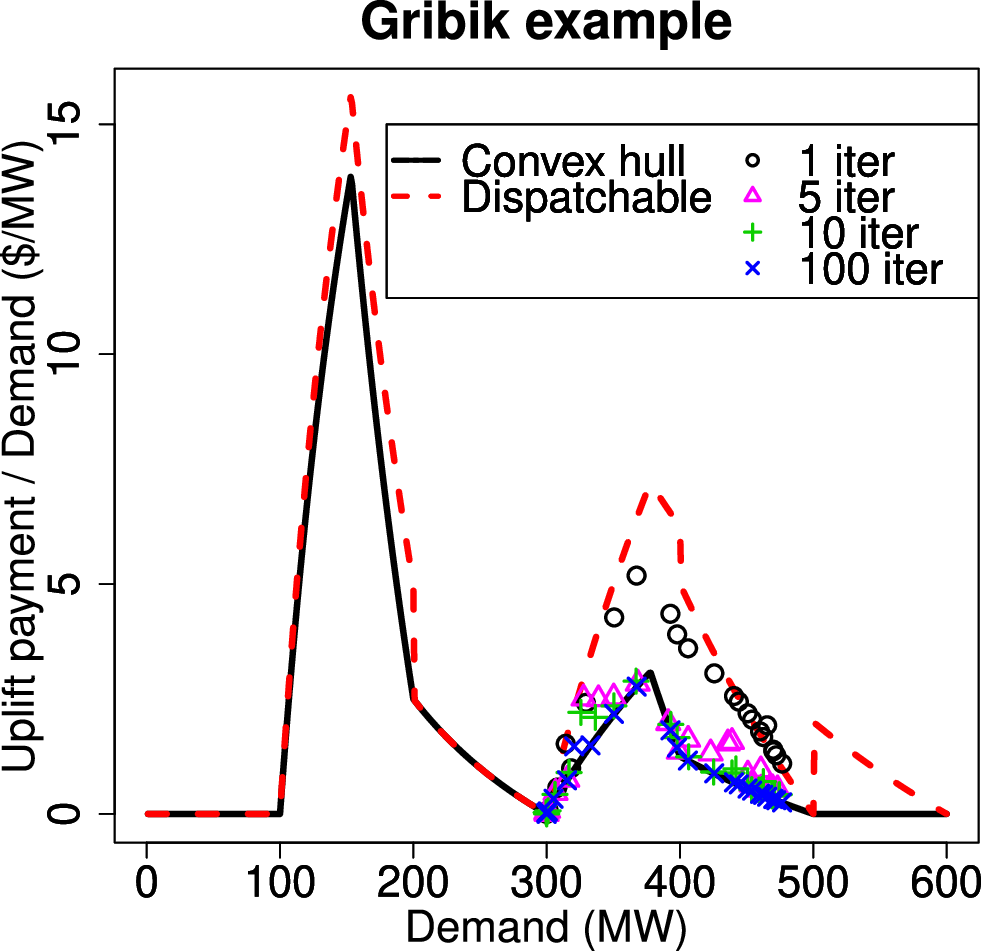}
    \end{minipage}
    \begin{minipage}{.49\columnwidth}
      \centering
      \includegraphics[width=\columnwidth,bb=0 23 475 480,clip]{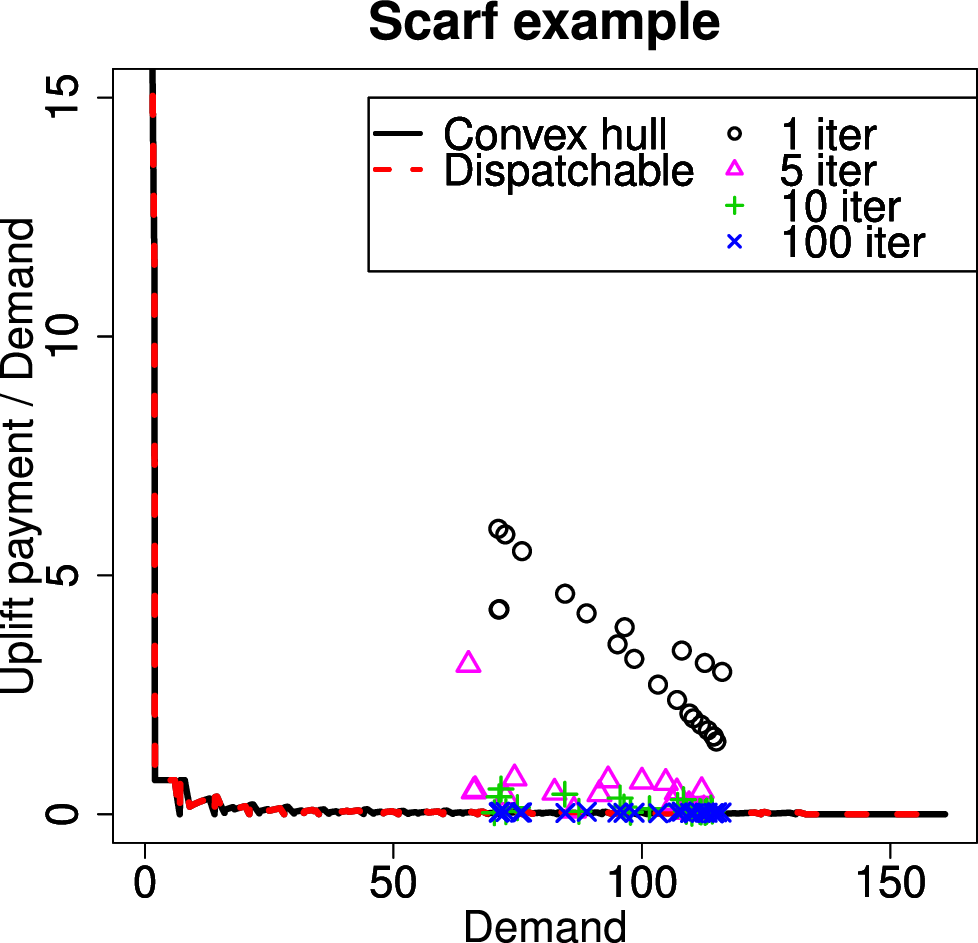}
    \end{minipage}
    \caption{Uplift payments at each iteration.}\label{iter}
  \end{figure}

  Fig. \ref{iter} 
  shows the uplift payments on the price after the $k$-th iteration 
  $\lambda^k_t~ (t=1,2,\dots,24,~k=1,5,10,100)$ of our algorithm (i.e., Algorithm \ref{subgradient}). 
  We can see that the prices of our algorithm achieve 
  lower uplift payments than those of the convex cost model 
  (i.e., `LMP' in Fig. \ref{CHP_LMP}) 
  in a few iteration. Furthermore, after about $10$ iterations, 
  the uplift payments reach sufficiently close to the solid line. 
  This implies that the prices $\lambda_t^k$ are almost the same as CHPs. 
  Under the pricing scheme that we assumed, the market participants have to 
  show prices or levels of consumption and production to each other at each iteration. 
  Therefore, it is unrealistic to iterate many times. 
  Our model shows smaller uplift payments than the convex cost model 
  within a realistic number of iterations. 
  
  \subsubsection{Change of the uplift payments with respect to computational time}
  \begin{figure}
    \begin{minipage}{.49\columnwidth}
      \centering
      \includegraphics[width=\columnwidth,bb=0 23 475 480,clip]{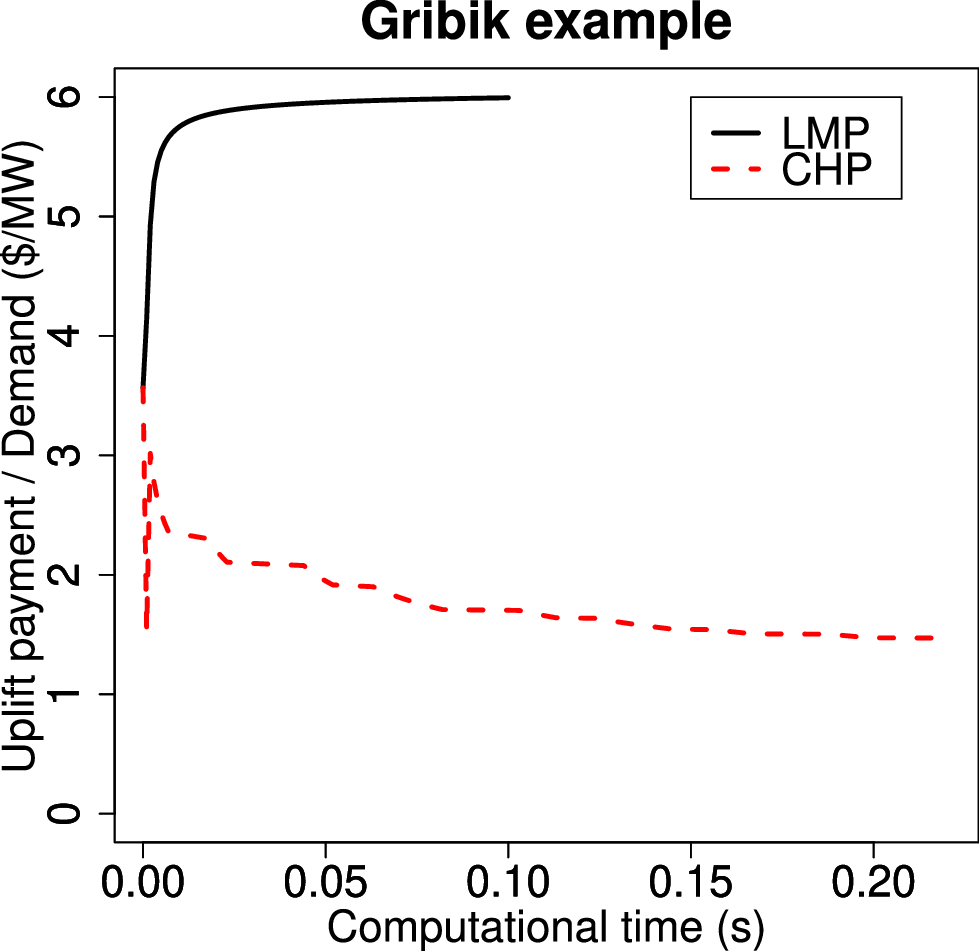}
    \end{minipage}
    \begin{minipage}{.49\columnwidth}
      \centering
      \includegraphics[width=\columnwidth,bb=0 23 475 480,clip]{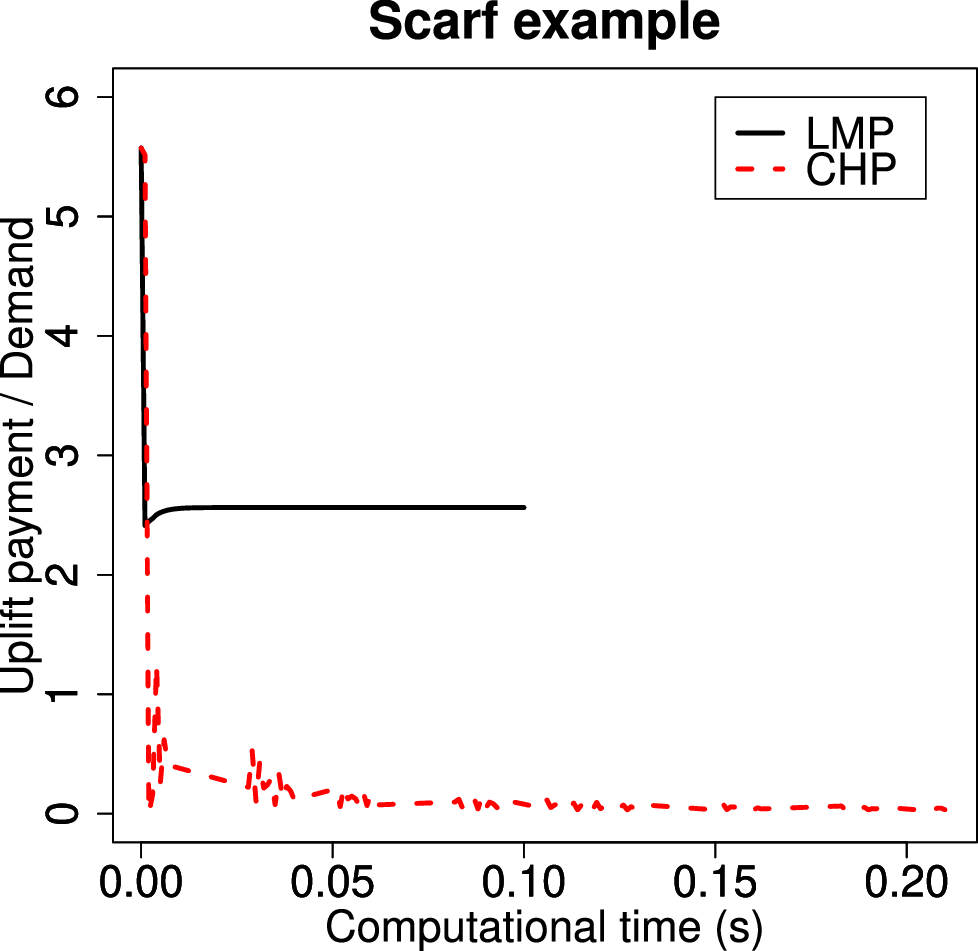}
    \end{minipage}
    \caption{Change of the uplift payments with respect to computational time for 100 iterations.}\label{time}
  \end{figure}
  While we only show the results at 1:00 (i.e., $t=2$) in Fig. \ref{time} 
  for lack of space, remaining time has similar numerical results. 
  On Gribik example, the convex cost model increases the uplift payments with computational time, 
  since the uplift payments for the initial price $\lambda^0_t$ 
  are smaller than ones for the LMPs. 
  By contrast, our model decreases the uplift payments with computational time. 
  On Scarf example, the both algorithms decrease the uplift payments. 
  The algorithm for the convex cost model (i.e., Algorithm 1 in \cite{Miyano2012a}) 
  reduces the uplift payments faster 
  than one for our model (i.e., Algorithm \ref{subgradient} in this paper) at first, 
  since it takes less computational time for an iteration. 
  However, we can see that 
  the algorithm for our model leads less uplift payments than 
  one for the convex cost model a short time later.

\vspace{-1pt}
\section{Conclusion}\label{Conclusion}
This paper provided a new dynamic pricing model based on  
the CHP approach which has not been used in the context of dynamic pricing. 
We first considered a nonconvex cost function within the settings of the UCP, 
and added it to a social welfare maximization problem. 
We proved that a solution of its dual problem (i.e., Lagrange multiplier) gives the CHP. 
This implies that our model minimizes the uplift payment for an equilibrium demand. 
Since our model itself is formulated as a mixed integer programming problem, 
and moreover, the objective function of our model would be nonsmooth, 
it is difficult to solve our model exactly. Therefore, 
we provided an iterative approximation algorithm based on the subgradient method. 
Numerical experiment showed our pricing model led to smaller uplift payment compared 
with existing LMP models with convex cost functions. In addition, our pricing algorithm 
reduced the uplift payments in a few iterations and a little computational time. 

In our numerical experiment, we used examples where generators have 
piece-wise linear cost functions. 
Thus we could use a mixed integer programming solver. 
We are planning to investigate ways to deal general nonlinear 
variable cost functions. We are also planning to extend our model to 
a multi-agent and multi-period one with network constraints as in \cite{Hogan2003}.


\ifCLASSOPTIONcaptionsoff
  \newpage
\fi



%


\vspace{-0.6cm}
\begin{IEEEbiographynophoto}{Naoki Ito}
received his B.E. degree in Administration Engineering 
from Keio University, Yokohama, Japan, in 2013. He is currently working toward the 
M.E. degree at Keio University, Yokohama, Japan. 
His research interests include optimization methods for 
nonconvex optimization problems. 
\end{IEEEbiographynophoto}
\vspace{-0.8cm}
\begin{IEEEbiographynophoto}{Akiko Takeda}
received the B.E. and M.E. degrees in
Administration Engineering from Keio University, Yokohama,
Japan, in 1996 and 1998, respectively, and the Dr.Sc. degree
in Information Science from Tokyo Institute of Technology,
Tokyo, Japan, in 2001. She is currently an Associate
Professor at the Department of Mathematical Informatics at
the University of Tokyo, Japan. Her research interests
include solution methods for decision making problems under
uncertainty and nonconvex optimization problems, which
appear in financial engineering, machine learning, energy
systems, etc.
\end{IEEEbiographynophoto}
\vspace{-0.8cm}
\begin{IEEEbiographynophoto}{Toru Namerikawa}
received the B.E., M.E and Ph. D of Engineering degrees 
in Electrical and Computer Engineering from Kanazawa University, 
Japan, in 1991, 1993 and 1997, respectively. 
He is currently an Associate Professor at Department of System Design Engineering, 
Keio University, Yokohama, Japan. 
He held visiting positions at Swiss Federal Institute of Technology in Zurich in 1998, 
University of California, Santa Barbara in 2001, 
University of Stuttgart in 2008 and Lund University in 2010. 
His main research interests are robust control, distributed and cooperative control 
and their application to power network systems. 
\end{IEEEbiographynophoto}

\end{document}